\documentclass[12pt]{amsart}
\makeindex
\usepackage[latin1]{inputenc}
 \usepackage[T1]{fontenc} \usepackage{lmodern}
\usepackage{epsfig,graphicx,color,amsmath,a4wide,amssymb,amsfonts,amsbsy,xy,
stmaryrd,latexsym,epsf,pstricks,verbatim,times,hyperref}

\usepackage{makecell}
\usepackage{pgf,color}
\usepackage{pgfarrows}

\xyoption{all}

\definecolor{Grey}{rgb}{.5,.5,.5}

\definecolor{LightBlue1}{rgb}{.2,.4,0.9}
\definecolor{LightBlue2}{rgb}{.3,.5,0.9}
\definecolor{LightBlue3}{rgb}{.4,.6,0.9}
\definecolor{LightBlue4}{rgb}{.5,.7,.9}
\definecolor{LightBlue5}{rgb}{.6,.8,.9}
\definecolor{LightBlue6}{rgb}{.7,.9,.9}

\definecolor{Red}{rgb}{.9,.0,.0}
\definecolor{LightRed1}{rgb}{0.9,.2,.4}
\definecolor{LightRed2}{rgb}{0.9,.3,.5}
\definecolor{LightRed3}{rgb}{0.9,.4,.6}
\definecolor{LightRed4}{rgb}{.9,.5,.7}
\definecolor{LightRed5}{rgb}{.9,.6,.8}
\definecolor{LightRed6}{rgb}{.9,.7,.9}

\xyoption{all}
\usepackage[english]{babel}

\newcounter{noalgo}[section]
\newdimen\indentalgo
\newdimen\indentalgodec\indentalgo=0.0mm\indentalgodec=10mm

\newcommand{\If}{\advance\indentalgo by \indentalgodec {\bf if }}

\newcommand{\For}{\global\advance\indentalgo by \indentalgodec {\bf for }}

\newcommand{\Endindent}{\global\advance\indentalgo by -\indentalgodec}

\newdimen\decalage \decalage=0.5cm
\newcounter{algo} 

\def\<<{\leavevmode
  \raise0.28ex\hbox{$\scriptscriptstyle\langle\!\langle$}\nobreak
  \hskip -.6pt plus.3pt minus.2pt\,}
\def\>>{\,\nobreak\hskip -.6pt plus.3pt minus.2pt
  \raise0.28ex\hbox{$\scriptscriptstyle\rangle\!\rangle$}}

\def\<<{\leavevmode
  \raise0.28ex\hbox{$\scriptscriptstyle\langle\!\langle$}\nobreak
  \hskip -.6pt plus.3pt minus.2pt\,}
\def\>>{\,\nobreak\hskip -.6pt plus.3pt minus.2pt
  \raise0.28ex\hbox{$\scriptscriptstyle\rangle\!\rangle$}}

\newtheorem{cor}{Corollary}

\newtheorem{theorem}{Theorem}
\newtheorem{lemma}{Lemma}

\def\ZZ{{\mathbb Z}}

\def\SSIG{{\mathop{\sum\!'}\nolimits}}

\providecommand{\myproofname}{Proof}

\begin{document}

\title{Bad witnesses for a composite number}

\author{Johnathan Djella Legnongo}
\address{Johnathan Djella, UFR of Science and Technology, Assane
Seck University of Ziguinchor, BP 523, Senegal.}
\email{johndjella@gmail.com}

\author{Tony Ezome}
\address{Tony Ezome, Laboratoire de Recherche en Math\'ematiques
et Applications (LAREMA),
\'Ecole Normale Sup\'erieure (ENS),
BP 17 009 Libreville, Gabon.}
\email{tony.ezome@gmail.com}

\author{Florian Luca}
\address{Florian Luca, School of Mathematics, University of
the Witwatersrand, Private Bag 3, Wits 2050, South Africa.}
\email{Florian.Luca@wits.ac.za}

\maketitle

\begin{abstract}
We describe the 
average sizes of the set of bad witnesses for
a pseudo-primality test which is the product
of a multiple-rounds Miller-Rabin test by the Galois test.\\
\end{abstract}


\vspace{.75cm}

\section{Introduction}

One of the simplest questions we have 
in mind, when handling 
a large number, is to know whether this number is prime or composite.
There are of course several methods
allowing to decide primality.
The most important algorithms are divided into two subfamilies:
primality tests and pseudo-primality tests.
We refer to \cite[Chapters 8 and 9]{Cohen1}, \cite{Ezome}
and \cite{Schoof}] for surveys on
known algorithms, from oldest to most recent.
Actually, we are talking about a very long-standing mathematical
problem which was already addressed by Euclid's Elements.
Closer to our time, the Fermat's little theorem gives rise to a
pseudo-primality test which studies
the primality of an odd integer $n$ by checking the congruence
$x^{n-1}\equiv 1 \pmod n$ whenever an integer $x$
relatively prime to $n$ is randomly chosen.
If the congruence is false, then $n$ is obviously composite.
When the congruence is true, one can only conclude that $n$ is probably prime.
Indeed, Carmichael numbers are composite
numbers for which the previous congruence is true
for any $x$. However,
Alford, Granville and Pomerance proved in
\cite{Alford-Granville-Pomerance} that there are
infinitely many Carmichael numbers.
This compromises the reliability of the Fermat test.
The most commonly used algorithm
in practice for prime detection is the Miller-Rabin pseudo-primality test.
This algorithm is an improvement of the Fermat test
resulting from the work by Artjuhov \cite{Artjuhov},
Miller \cite{Miller} and Rabin \cite{Rabin}.
Setting $n-1 = 2^km$ with $m$ odd, we say that $n$ passes a
Miller-Rabin test if one of the congruences
 \begin{equation}\label{eq:1}
 x^m \equiv 1 \pmod n, \ \text{ or } \
x^{2^im} \equiv -1 \pmod n \ \text{ for some } \  i<k
\end{equation}
holds with an integer $x$ relatively prime to $n$ chosen randomly.
Such an $x$ is a {\it witness} of the pseudo-primality of $n$
with respect to the Miller-Rabin test. One says that $x$ is a
\textit{ bad witness} if it satisfies
one of the above congruences 
while $n$ is in fact composite.
The density of bad witnesses is an important characteristic
for a pseudo-primality test as it measures  the reliability of this test.
It is proved 
\cite[Proof of theorem 2.1]{Schoof} that
the Miller-Rabin test is very effective in the case when 
$n$ has many prime divisors.
A pseudo-primality test called the \textit{Galois test}
which is efficient when the integer
to be tested has only very few prime divisors
has been constructed in \cite{Couv.Ezo.Ler}.
At the end of that paper,
the authors deduced a stronger pseudo-primality test
which is the product of \textit{a multiple-rounds Miller-Rabin test}
($i.e$ running several Miller-Rabin tests at the same time)
by a Galois test. This product test
takes advantages of strengths of each of its components,
especially in the extreme cases when the integer to be tested
has either many or very few prime divisors.
However, nothing is known yet about the average case analysis.
A first step in this direction has been made by Erd\"os and Pomerance
in \cite{Erdos-Pomerance}, they
focused on the special case of a one-round Miller-Rabin test.

The present paper describes the average
sizes of the set $\mathbf{Str}(n)$ 
of bad witnesses of an odd number $n$ with respect
to the stronger test proposed in \cite{Couv.Ezo.Ler}.
We start by studying the set $\mathbf{Gal}(n)$ of bad witnesses
for the Galois test in Section \ref{sect:2}.
We recall some well known results
concerning $\mathbf{Gal}(n)$ in \ref{subsect:1}. Then,
we focus on the arithmetic and geometric
mean values of its cardinality denoted by
$Gal(n)$ in Subsections \ref{subsect:2} and \ref{subsect:3}.
Section \ref{sect:3} is devoted to the stronger test.
We first describe the average numbers of bad witnesses
for a multiple-rounds Miller-Rabin test. And then, we specify the
case of the stronger test.

\subsection*{Acknowledgments}
This study has been carried out with financial support from the French State,
managed by CNRS in the frame of the \textit{Dispositif de Soutien aux
Collaborations avec l'Afrique subsaharienne}
(via the REDGATE Project and the IRN AFRIMath).
The first two authors were
supported by Simons Foundation via the PREMA project.
The first author is grateful to EMS-Simons for Africa
for their support to his PhD education.\\

\section{Bad witnesses of the Galois test}\label{sect:2}

In this section, we compute functions that bound from above and from
below the arithmetic
mean of the number of bad witnesses for the Galois test,
we also specify its geometric mean.
We start by recalling some known results.

\subsection{Preliminaries}\label{subsect:1}

Let $n\ge 3$ be an odd integer to be tested, and
$n=\prod_{p^{v_p}\| n} p^{v_p}$
its prime factorization.
Let $S$ be a $d$-dimensional cyclic extension of
$R:={\mathbb Z}/n{\mathbb Z}$ with Galois group generated by $\sigma$.
As a ring $S$ is isomorphic to
$$
\prod_{p\mid n} S_p,
$$
where $S_p:=S/p^{v_p}S$. Fixing a prime factor
$p$ of $n$, we set ${\bf L}_p:=S_p/pS_p$ and $\mathbf{K}:=R/pR$.
It is known \cite[Section 2]{Couv.Ezo.Ler}
that $S$ is a free $R[\sigma]$-module of rank $1$.
We denote by $(\sigma^i(\omega))_{0\le i\le d-1}$ a
normal basis of $S$ over $R$. The ring
${\bf L}_p$ has only a finite number
of prime ideals, say ${\bf p}_1,\ldots,{\bf p}_m$.
The residue fields  ${\bf L}_p/{\bf p}_i$ for $i=1,\ldots,m$
are all isomorphic to ${\mathbb F}_{p^f}$, the finite field
with $f$ elements, where of course 
$d=fm$. In particular, the ring ${\bf L}_p$ is isomorphic to
$$
\prod_{p\mid n} {\mathbb F}_{p^f}.
$$
So, the more composite $d$ is the more possibilities
for the pair $(m,f)=(m,d/m)$.
The $R$-automorphism
$\sigma : S \longrightarrow S$ induces
a $\mathbf{K}$-automorphism of ${\bf L}_p$.
There is an integer $z$ coprime to $m$ such that
%$\llbracket  1,f-1\rrbracket$ and satisfying
$$x^p=\sigma^{zm}(x), \ \ \text{ for all } x\in {\bf L}_p.$$
We denote by $t\in \llbracket  1,f-1\rrbracket$ the inverse of 
$z$ modulo $f$ (if $f=1$, we have $z=t=0$).
Note that the integers $f,m$ and $t$ depend on the prime $p$.
For a fixed dimension $d$, there are only finitely many possibilities for $(m,f,t)$, and every prime  gets some such possibility assigned to it (these possibilities might not all be the same they may vary with the primes). 

Assume that $\sigma(\omega)=\omega^n$. Then
$n$ passes successfully a Galois test of dimension $d$
if by choosing randomly a nonzero element $x$ in $S$,
we have that $x$ is invertible and
\begin{equation}\label{eq:16}
\sigma(x)=x^n.
\end{equation}
Let $\mathbf{Gal}(n)$ be the set of invertible
elements in $S^\times$ which are solutions of
Equation $(\ref{eq:16})$. We denote by $\mathrm{Gal}(n)$ its cardinality.
It is proved \cite[Section 1 and Section 2]{Couv.Ezo.Ler}
that
\begin{equation}\label{eq:17}
Gal(n)=\prod_{p\mid n}\gcd(n^m-p^t,p^f-1).
\end{equation}

Since $n^m-p^t$
divides $n^{mf}-p^{tf}=(n^d-1)-((p^f)^t-1)$,
we deduce that ${Gal}(n)$ divides
$$D(n):=\prod_{p\mid n} \gcd(p^f-1,n^d-1).$$
Note that the latter quantity
counts the number of elements $x$ in the group of invertible
elements of the ring 
$$
\prod_{p\mid n} {\mathbb F}_{p^f}
$$
satisfying $x^{n^d-1}=1$. Since these elements form a subgroup,
there exists some positive integer $k$ such that
\begin{equation}\label{eq:17}
D(n)=\frac{1}{k} \prod_{p\mid n} (p^f-1).
\end{equation}
Set $L(x)= \exp\left(\frac{\log x\log \log \log x}{\log \log x}\right)
\text{ for all large } x$. The lemma below is very useful.

\begin{lemma}\label{lemma:1}
Given a positive integer $k$, the number of composite numbers $n\le x$
having a prime divisor $p>kL(x)$ and such that
$D(n)=\frac{1}{k} \prod_{p\mid n} (p^f-1)$ is bounded
from above by $ xL(x)^{-1+o(1)}$.
\end{lemma}

\begin{proof} 
Indeed, equation $(\ref{eq:17})$ implies that
$\prod_{p\mid n}(p^f-1)$ divides $k(n^d-1)$. So,
$$
p-1\mid k(n^d-1).
$$
Set $D:=(p-1)/\gcd(k,p-1)$. Since $p$ is fixed and so is $k$,
the number $D$ is fixed. We need to count the $n\le x$ 
divisible by $p$ such that $n^d-1$ is divisible by $D$.
Write $D=\prod_{q^u\| D} q^u$.
Then the number of solutions $y$ to the congruence
$$
y^d-1\equiv 0\pmod D
$$
is equal modulo $D$ to 
$$
\prod_{q^u\| D} \rho(q^u),
$$
where for a prime power $q^j$, $\rho(q^j)$ denotes the number of
solutions $y$ modulo $q^j$ of the congruence $y^d-1\equiv 0\pmod q^j$.
Therefore, the number of $n\le x$ 
divisible by $p$ such that $n^d-1$ is divisible by $D$ is 
bounded from above by $C_d d^{\omega(D)}$, where $C_d$ depends only on $d$.
This upper bound is obvious for $q>d$, since
$y^d-1\equiv 0\pmod q$ has then at most $d$ solutions
modulo $q$ which by Hensel's lemma are each extendible in a
unique way to a  solution $x$ modulo $q^j$ for every $j\ge 1$.
The constant term $C_d$ obviously equals $1$ in this case.
A nontrivial $C_d$ appears when dealing with moduli which are small
prime powers of primes dividing $d$. Since
\begin{eqnarray*}
C_d d^{\omega(D)} & \le & C_d d^{O(\log(p-1)/\log\log(p-1))}\\
& = & \exp\left(\log C_d+O\left(\frac{\log x}{\log\log x}\right)\right)\\
& = & L(x)^{o(1)},
\end{eqnarray*}
the number of such progressions is $L(x)^{o(1)}$.
Let us count the $n\le x$ in a fixed progression.
We take $x_0\in [1,D-1]$ such that $x_0^d\equiv 1\pmod D$
and we count the number of $n\le x$ such that $n\equiv x_0\pmod D$. Since $n\equiv 0\pmod p$ and 
$(p,D)=1$ (since $D\mid p-1)$, we get that this puts $n$ into some progression $n\equiv x_0'\pmod {pD}$. The number of such $n\le x$ is 
\begin{equation}\label{eq:19}
1+\left\lfloor \frac{x}{pD}\right\rfloor.
\end{equation}
Since
$$
\frac{x}{pD}=\frac{x}{p(p-1)/\gcd(k,p-1)}\le \frac{kx}{p(p-1)},
$$
we have
$$
\sum_{p>kL(x)} \left\lfloor \frac{x}{pD}\right\rfloor \le 2kx\sum_{p>kL(x)}\frac{1}{p^2}\ll \frac{x}{L(x)}. 
$$
It remains to count the $1$'s in equation $(\ref{eq:19})$.
These are the initial
terms $x_0'$ of the solutions modulo
$Dp$, namely the smallest positive integer $n\le x$ such
that $n\equiv x_0\pmod D$ and $n\equiv 0\pmod p$. Write 
$n=pm$. Then $n^d\equiv p^d m^d\equiv 1\pmod D$.
Since $D\mid p-1$, we get that $m^d\equiv 1\pmod D$.
This means $D\mid m^d-1$. We know that $m>1$ since $n$
is not prime. Thus, $m>D^{1/d}$. Since $pm\le x$,
it follows that $pD^{1/d}\le x$.

Let us make a parenthesis and take a closer look at $\gcd(p-1,k)$. Let $k_1=\gcd(p-1,k)$. We first treat the case $k_1\ge L(x)$.  Note that $k\mid \prod_{p\mid n} (p^f-1)<n^f\le x^d$. Thus, $k_1$ is a divisor of 
a number of size $x^d$. The number of divisors of $k$ is therefore at most 
$$
\max_{k<x^d} \tau(k)=\exp((O(\log x^d/\log \log x^d))=\exp(O(d\log x/\log\log x))=L(x)^{o(1)}
$$
as $x\to\infty$. Assume now that $k_1\mid k$ is fixed and let us count primes $p\le x$ such that $\gcd(p-1,k)=k_1$. Then $p\equiv 1\pmod {k_1}$. 
We bound the number of such primes $p\le x$ trivially $\le x/k_1\le x/L(x)$. Given $p$, $D$ is a divisor of $p-1$, so $D$ can have at most $\tau(p-1)=L(x)^{o(1)}$ possibilities. Thus, for fixed $k_1$, there are 
$x/L(x)^{1+o(1)}$ possibilities for the pair $(p,D)$. Summing up over the $L(x)^{o(1)}$ possible values of $k_1$ (divisors of $k$ larger than $L(x)$), we get a count of $x/L(x)^{1+o(1)}$ for the number of such pairs $(p,D)$.

Next, we assume that $k_1<L(x)$.  Then
$$
D=\frac{p-1}{\gcd(p-1,k)}=\frac{p-1}{k_1}>\frac{p-1}{L(x)},
$$ 
so 
$$D^{1/d}\ge (p-1)^{1/d}/L(x)^{1/d}.
$$ 
Thus, $p(p-1)^{1/d}\le pD^{1/d}\le xL(x)^{1/d}=x^{1+o(1)}$, showing that
$p\le x^{d/(d+1)+o(1)}$. Given $p$, $D$ is determined
in at most $\tau(p-1)=L(x)^{o(1)}$ ways, so the modulus 
$Dp$ is determined in at most $x^{d/(d+1)+o(1)}$ ways. Thus, 
the number of starting points $x_0\pmod D$ is also determined in at most $L(x)^{o(1)}$ ways. So, the number of $1$'s in this last case is in fact much smaller namely at most
$$
x^{d/(d+1)+o(1)}<\frac{x}{L(x)}
$$
for large $x$. This concludes the proof of the lemma.
\end{proof}

\subsection{Average order of $\mathrm{Gal}(n)$}\label{subsect:2}

Here, we study the arithmetic mean
of the number of bad witnesses of the Galois test.
We first compute a function which bounds it from below. In what follows
$\SSIG$ denotes a sum over composite numbers.

\begin{theorem}\label{thm:5}
For all large $x$,
$$
\frac{1}{x}\SSIG_{n\le x} Gal(n)> x^{\frac{15}{23}}.
$$
\end{theorem}

\begin{proof}
Let $M(x)$ denote the least common multiple
of the integers up to $x$.
For any $y$, let $$\mathbf{P}(y,x)=\{ p\le y : p-1|M(x) \}$$
It is shown (\cite{Erdos} or \cite{Pomerance})  that
there is a real number $\alpha>1$
 such that
$$\# \mathbf{P}(x^{\alpha'},x)=O(x^{\alpha'}/\log x) \quad \text{for all }
0< \alpha'< \alpha.$$
We let 
$$\beta=\sup \{ \alpha : \#  \mathbf{P}(x^{\alpha'},x) = x^{\alpha'+o(1)}
\text{ for all }  0< \alpha'<\alpha \}.$$
From \cite{Balog}, we have
$$\beta >23/8.$$
Let $L$ denote an upper bound for Linnik's constant, so that given 
positive integers $a,m$ with $\gcd(a,m)=1$ and $m>1$, then there is a prime
$p\equiv a\pmod m$ with $p<m^L$.
Let $\alpha$ be such that $1<\alpha<\beta$ and
$0<\epsilon<\alpha-1$ arbitrarily small.
We set 
$$\begin{array}{c}
M=M(\log x/\log \log x),\\
\mathbf{P}=\mathbf{P}(\log^\alpha x,\log x/\log \log x)-
\{  p : p\le \log^{\alpha-\epsilon}x \}.
\end{array}
$$
Let $\mathbf{S}$ denote the set of integers composed of exactly
$$k=[\log(x/M^L)/\log(\log^\alpha x)]$$
distinct primes in $\mathbf{P}$.
Thus, if $s\in \mathbf{S}$, then
\begin{equation}\label{eq:14}
x^{1-\epsilon} <s<x/M^L.
\end{equation}
Let $\mathbf{S}'$ be the set of products $sq$,
where $s\in \mathbf{S}$ and $q$ is the least prime such that 
$$sq\equiv 1 \pmod M.$$
It is shown (\cite{Erdos-Pomerance}, Proof of Theorem 2.1)
that 
\begin{equation}\label{eq:15}
 sq\le x, \quad \text{ and } \quad \#\mathbf{S}' \ge x^{(\alpha-1)\alpha^{-1}+o(1)}.
\end{equation}
So if $n=sq \in \mathbf{S}'$, then 
$$\mathrm{Gal}(n)=\prod_{p|n} \gcd(n^m-p^t,p^f-1)\ge \prod_{p|s}\gcd(n^m-p^t,p^f-1)
\ge \prod_{p|s}\gcd(n-1,p-1)$$
$$=\prod_{p|s}\gcd(M,p-1)=\varphi(s).$$
Theorem $328$ in \cite{Hardy-Wright} implies that
$$s/\log\log s=O(\varphi(s)).$$
Besides, $(\ref{eq:14})$ and $(\ref{eq:15})$ 
give 
$$s/\log\log s\ge x^{1-\epsilon}/\log\log x.$$ So
$$\frac{1}{x} \SSIG_{n\le x } Gal(n)
 \ge \frac{1}{x} \sum_{n \in \mathbf{S}'}
Gal(n)\ge x^{-\epsilon + o(1)} \cdot \#\mathbf{S}'
\ge x^{1-\varepsilon-\alpha^{-1}+o(1)}.$$
 Since $\varepsilon>0$ is arbitrarily small and $\alpha$ is arbitrarily
 close to $\beta$, we have
$$\frac{1}{x}\SSIG_{n\le x} Gal(n)\ge x^{1-\beta^{-1}+o(1)}.$$
Hence,
$$
\frac{1}{x}\SSIG_{n\le x} Gal(n)> x^{\frac{15}{23}+o(1)},
$$
since $\beta>23/8$.

\end{proof}

We now compute a function which bounds from above
the number of bad witnesses of the Galois test.

\begin{theorem}\label{thm:4}
As $x \to\infty$,
\begin{equation}\label{eq:18}
\frac{1}{x}\SSIG_{n\le x} Gal(n)\le x^{d}L(x)^{-1+o(1)}.
\end{equation}
\end{theorem}

\begin{proof}
We saw in Subsection \ref{subsect:1} that
$Gal(n)$ divides $D(n)$. Therefore, it suffices to show that
$$
\frac{1}{x}\SSIG_{n\le x} D(n)\le x^{d}L(x)^{-1+o(1)}.
$$
For every integer $k\ge 1$, we let $\mathbf{C}_k(x)$ be the set of composite
numbers $n\le x$ such that $D(n)=\frac{1}{k} \prod_{p\mid n} (p^f-1)$, and
we set $C_k(x)= \#\mathbf{C}_k(x)$. Then
$$\SSIG_{n\le x} D(n) = \sum_k \sum_{n\in  \mathbf{C}_k(x)}
\frac{1}{k} \prod_{p\mid n} (p^f-1)
\le \sum_{n\le x} \frac{n^{d}}{L(x)} +
\sum_{k\le L(x)} \frac{1}{k}\sum_{n\in  \mathbf{C}_k(x)} n^d.$$
So,
$$\SSIG_{n\le x} D(n)\le \frac{x^{d+1}}{L(x)}+
x^{d}\sum_{k\le L(x)} \frac{C_k(x)}{k}.$$

Therefore, we have to show that
$$
C_k(x)\le xL(x)^{-1+o(1)} \text{ for } k\le L(x).$$
There are three cases to be considered:
\begin{itemize}
\item[(i)] $n<x/L(x)$,
\item[(ii)] $n$ is divisible by some prime 
$p>kL(x)$,
\item[(iii)] $n\ge x/L(x)$ and every prime divisor of $n$ is
at most $kL(x)$.
\end{itemize}
Situation $(i)$ is trivial.
Situation $(ii)$ is taken care of by Lemma \ref{lemma:1}. 
Situation $(iii)$ implies that $n$ has a divisor $n_1$ satisfying
\begin{equation}\label{eq:20}
\frac{x}{kL(x)^2} < n_1 <\frac{x}{L(x)}.
\end{equation}
By $(\ref{eq:17})$, we know that
the exponent $\lambda_S(n)$ of the group
$\prod_{p\mid n} {\mathbb F}_{p^f}^*$ divides $k(n^d-1)$. 
We would like to follow the proof of case (iii) of Theorem 2.2 in \cite{Erdos-Pomerance}, but there are additional complications due to the possible prime factors of $n$ which might appear 
to powers larger than $1$ in the factorization of $n$. So, we write $n_1=ab$, where $a$ is squarefree and $b$ is squarefull. We may assume that $b\le L(x)^2$, since 
if not, $b>L(x)^2$ and the number of $n\le x$ such that $b\mid n$ is $\le x/b$. Since 
$$
\sum_{\substack{m>y\\ m~{\text{\rm squarefull}}}} \frac{1}{m}=O\left(\frac{1}{\sqrt{t}}\right),
$$
we get that the set of such $n$ is of cardinality 
$$
O\left(\frac{x}{L(x)}\right)=\frac{x}{L(x)^{1+o(1)}}
$$
irregardless of the value of $k$. So, assuming $b\le L(x)^2$, we get that 
\begin{equation}
\label{eq:a}
\frac{x}{k L(x)^4}\le \frac{x}{k L(x)^2 b}\le a\le \frac{x}{L(x) b}.
\end{equation}
We fix a squarefree $a$, a squarefull $b$ coprime to $a$ and smaller than $L(x)^2$ and we look for $n=b\lambda$. Clearly, $\lambda\le x/b$. Further, $a\mid \lambda$ and additionally $\lambda(a)\mid k(n^d-1)$, 
which shows that $\lambda^d\equiv (b^*)^d\pmod {\lambda(a)/(\lambda(a),k)}$. Here, $b^*$ is the inverse of $b$ modulo $\lambda(a)/(k,\lambda(a))$. Since $b$ is fixed, this puts  $\lambda\le x/b$ in at most 
$C_d d^{\omega(\lambda(a)/(k,\lambda(a)))}$ arithmetic progressions modulo $\lambda(a)/(k,\lambda(a))$. By the Chinese Remainder Theorem, this puts $\lambda\le x/b$ in 
$$
C_d d^{\omega(\lambda(a)/(k,\lambda(a)))}
$$
progressions modulo $a\lambda(a)/(k,\lambda(a))$. The number of such numbers $\lambda\le x/b$ is 
$$
\le C_d d^{\omega(\lambda(a)/(k,\lambda(a)))} \left(\frac{x}{b a\lambda(a)/(k,\lambda(a))}+1\right).
$$
Since $\lambda(a)\le x$ has $O(\log x/\log\log x)$ prime factors, the factor multiplying the parenthesis is of size $\exp(O(\log x/\log_2 x))=L(x)^{o(1)}$ as $x\to\infty$. Factoring out such factor and summing up the remaining parenthesis over all 
numbers of the form $n_1=ab$ we get a count of 
$$
\sum_{b\le L(x)^2} \sum_{a}^* \frac{x(k,\lambda(a))}{ab \lambda(a)}+\sum_{ab\le x/L(x)} 1.
$$
Here, for a fixed $b$, the star on the inner summation indicates that the summation is over the squarefree $a$'s satisfying \eqref{eq:a}. 
The second sum is the number of positive integers up to $x/L(x)$ so it is $O(x/L(x))$. The first inner sum is for fixed $b$ estimated as on the bottom of page 283 and top of page 284 in \cite{Erdos-Pomerance}. For a fixed $b$ it ends up being 
$$
\frac{x (\log x)^2}{bL(x/(kL(x)^2b)} 2^{(1+o(1))\log x/\log_2 x}.
$$
Since $b\le L(x)^2$ and $k\le L(x)$, we get that $L(x/(kL(x)^2b))=L(x)^{1+o(1)}$, while $2^{(1+o(1))\log x/\log_2 x}=L(x)^{o(1)}$ as $x\to\infty$. Hence, the first sum simply becomes 
$$
\frac{x}{L(x)^{1+o(1)}} \sum_{b\le L(x)^2} \frac{1}{b}\ll \frac{x}{L(x)^{1+o(1)}},
$$   
which is what we wanted.
\end{proof}

\subsection{Geometric mean value of $\mathrm{Gal}(n)$}\label{subsect:3}

Theorem 3.1 in \cite{Erdos-Pomerance} studied the geometric mean value of 
\begin{equation}\label{eq:21}
F(n)=\prod_{p\mid n} \gcd(p-1,n-1).
\end{equation}
The result there is  that 
$$
\left(\prod_{n\le x} F(n)\right)^{1/x}=(c_2+o(1)) (\log x)^{c_1}
$$
with \begin{equation}\label{eq:5}
c_1=\sum_d \frac{\Lambda(d)}{d\varphi(d)}, \qquad c_2=\exp\left(
1+\sum_d \frac{\Lambda(d)C(d)}{d} \right),
\end{equation}
where $\Lambda(n)$ denotes the von Mangoldt's function
and $\varphi(n)$ denotes the Euler's $\varphi$-function.
Since 
$$
{\text{\rm Gal}}(n)=\prod_{p\mid n} \gcd(n^m-p^t,p-1)\ge F(n),
$$
it follows that the geometric mean value of ${\text{\rm Gal}}(n)$ is at least as large as the geometric mean value of $F(n)$. Since $m,t$ depend on the prime $p$, it seems hard to 
find an asymptotic for the geometric mean of ${\text{\rm Gal}}(n)$. However, since ${\text{\rm Gal}}(n)$ divides 
$$H(n)=\prod_{p[\mid n} \gcd(p^d-1,n^d-1),
$$ 
it follows that the geometric mean of ${\text{\rm Gal}}(n)$ is at most as large as the geometric mean of $H(n)$. Concerning this last one we have the following result.

\begin{theorem}\label{thm:1}
The estimate
$$
\frac{1}{x}\sum_{n\le x} \log H(n)=c_3 \log\log x+O(d^4),
$$
holds with
$$
c_3=c_3(d):=\sum_{d_1\mid d}  \sum_{(\lambda(s),d)=d_1} f(s,d_1)^2 \frac{\Lambda(s)}{s\phi(s)}.
$$
In the above, $\lambda(s)$ is the Carmichael function of $s$, and a for a prime power $s$ and $d_1\mid \lambda(s)$, $f(s,d_1)$ denotes the number of elements in $({\mathbb Z}/s{\mathbb Z})^*$ of order exactly $d_1$.   Note that when $s$ is a prime power we have $f(s,d_1)=d_1$ except if $s=2^a$ with $a\ge 3$ and $d_1\ge 2$ in which case $f(s,d_1)=2d_1$. 
\end{theorem}

One may ask how big is $c_3(d)$? Well since $s\asymp \phi(s)\asymp \lambda(s)$ and $f(s,d_1)\asymp d_1$ when $s$ is a prime power, it follows that for a fixed $d_1$, 
$$
\sum_{(\lambda(s),d)=d_1} \frac{\Lambda(s)}{s\phi(s)}\ll  \sum_{t\ge 1} \frac{\log (d_1t)}{(d_1t)^2}\ll \frac{\log d_1}{d_1^2},
$$
showing that
$$c_3\ll \sum_{d_1\mid d} \log d_1=\log \left(\prod_{d_1\mid d} d_1 \right)= \log(d)^{\tau(d)/2}=(\tau(d)/2)\log d,
$$
so 
$$
c_3=O(\tau(d)\log d).
$$
In particular, $c_3=d^{o(1)}$ as $d\to\infty$.

\begin{proof}
We have
\begin{eqnarray*}
\sum_{n\le x} \log H(n) & = & 
\sum_{n\le x}\log \prod_{p\mid n}\gcd(n^d-1,p^d-1) \\
& = & \sum_{n\le x} \ \sum_{p\mid n} \ \sum_{s\mid \gcd(n^d-1,p^d-1)} \Lambda(s)\\
& = & \sum_{s\le x^d} \Lambda(s)\sum_{\substack{p\le x\\ s\mid p^d-1}} \sum_{\substack{n=p\ell\le x\\ s\mid (p\ell)^d-1}} 1\\
& = & \sum_{s\le x^d} \Lambda(s)\sum_{\substack{p\le x\\ p^d\equiv 1\pmod s}} \sum_{\substack{\ell\le x/p\\ \ell^d\equiv 1\pmod s}} 1.
\end{eqnarray*}
In the above, we wrote the condition $p\mid n,~n\le x$ as $n=p\ell\le x$, so $\ell\le x/p$. We first deal with large values of $s$. Let $S_1$ be the sum corresponding to $s=q^{\lambda}$, 
where $q\ge x$. Firstly, $\lambda\le d$. Secondly, if such $q$ appears then $q\mid p^d-1$ for some $p\le x$. Further, since $p\le x$ and $q>x$, it follows that the above congruence has only at most $d$ solutions $p$ altogether. Thus, for each such $s$, there are at most $d$ occurrences of $p$ such that $p^d\equiv 1\pmod s$, and the same goes for the $\ell$'s. For each such $q$ let $s_q$ be the maximal power of $q$ for which $s_q$ appears in $S_1$. Then the product of $s_q$'s divides
$$
\prod_{p\le x} (p^d-1) \le \left(\prod_{p\le x} p\right)^d=\exp((1+o(1))x),
$$ 
where the last equality follows from the Prime Number Theorem. Thus,\\

\qquad \qquad \qquad \qquad$
\sum_{q}' \log s_q=O(d x),
$\\

\noindent where $\sum'$ means that we are only summing over the $q$'s that appear in $S_1$. Since the exponents of such $q$ in $s_q$ is at most $d$ (so there are at most $d$ values of $s$ 
dividing $s_q$ appearing in $S_1$) and since for each such $s$ there are at most $d^2$ pairs $(p,\ell)$ such that $p^d\equiv \ell^d\equiv 1\pmod s$, we get that 
$$
S_1=O(d^4 x).
$$
From now on, we look at $s=q^{\lambda}$ but $q<x$. Let $S_2$ be the sum corresponding to $s=q^{\lambda}\ge x$. The equation $p^d\equiv 1\pmod s$ has $O(d)$ solutions and the same is true for 
the equation $\ell^d\equiv 1\pmod s$. Thus, given $s$, there are at most $O(d^2)$ pairs $(p,\ell)$ with $p\le x,~\ell\le x$ such that $p^d\equiv \ell^d\equiv 1\pmod s$. 
Since $s=q^{\lambda}\in (x,x^d]$, we have $\lambda\in (\log x/\log q, d\log x/\log q]$. Thus, there are $O(d\log x/\log q)$ possibilities for the exponent $\lambda$ in $s$ once $q$ is fixed. Hence,
$$
S_2=O\left(d^3\sum_{q\le x} \log q \left(\frac{\log x}{\log q}\right)\right)=O(d^3 \pi(x)\log x)=O(d^3 x).
$$ 
From now on, we may assume that $s\le x$. We now look at the condition $p^d\equiv 1\pmod s$. But we also have $p^{\lambda(s)}\equiv 1\pmod s$, where $\lambda(s)$ is the Carmichael 
function of $s$. Thus, $p^{\gcd(d,\lambda(s))}\equiv 1\pmod s$. This suggests putting $d_1=\gcd(d,\lambda(s))$ and studying $p^{d_1}\equiv 1\pmod s$. The unit group modulo $s$ is cyclic for 
primes powers $s$ except when $s=2^a$ with $a\ge 3$, in which case it is isomorphic ${\mathbb Z}/2{\mathbb Z}\times {\mathbb Z}/2^{a-2}{\mathbb Z}$.  It thus follows that the  number of residue classes modulo $s$ say $y$ such that $y^{d_1}\equiv 1\pmod s$ is exactly $d_1$ except if $s=2^a$ for $a\ge 3$ and $d_1>1$ in which case it is $2d_1$. Let this number be $f(s,d_1)$ like in the statement of the theorem and let these residues be $y_{i}(s,d_1)$ for $i=1,\ldots,f(s,d_1)$. We then have that $p^d\equiv 1\pmod s$ forces $p\equiv y_i(s,d_1)\pmod s$ for some 
$i=1,\ldots,f(s,d_1)$. For each such $p\le x$, the equation $\ell^d\equiv 1\pmod s$ also has 
exactly $f(s,d_1)$ solutions $\ell$ modulo $s$ and since $\ell\le x/p$, the number of such solutions is 
$$
f(s,d_1)\left(\left\lfloor \frac{x}{ps}\right\rfloor+O(1)\right).
$$
When $ps>x$, the first term in integer part is not present. We remove integer parts  and include the fractional parts into the $O(1)$ terms so the remaining sum is now
$$
\sum_{d_1\mid d} \sum_{\substack{s\le x \\ \gcd(\lambda(s),d)=d_1}} \Lambda(s) f(d_1,s)\sum_{i=1}^{f(s,d_1)} \sum_{\substack{p\le x\\ p\equiv y_i(s,d_1)\pmod s}} \left(\frac{x}{ps}+O(1)\right)=S_3+O(S_4),
$$
where $S_3$ is the sum involving the terms $x/ps$ and $S_4$ is the sum involving the $1$'s. For $S_4$, since $ps>x$, each class $y_i(s,d_1)$ contains at most one such prime $p$. 
So, 
$$
S_4\le \sum_{d_1\mid d} \sum_{\substack{s\le x\\ \gcd(\lambda(s), d_1)=d_1}} \Lambda(s) f(s,d_1)^2=O\left(d^2\sum_{s\le x} \Lambda(s)\right)=O(d^2 x).
$$
It remains to deal with $S_3$. For this we proceed as in \cite{Erdos-Pomerance} and split into $S_{3,1}$ and $S_{3,2}$ with $S_{3,1}$ being the sum over small $s$ (say $s\le (\log x)^2$) 
and $S_{3,2}$ being the sum over large $s$ (say $s>(\log x)^2$). As in that proof, using results of Norton and Pomerance, we get that for
\begin{equation}
\label{eq:NP}
\sum_{\substack{p\le x\\ p\equiv y_i(s,d_1)\pmod s}} \frac{1}{p} =\frac{\log_2 x}{ s\phi(s)}+O\left(\frac{\log s}{\phi(s)}\right)
\end{equation}
uniformly in $s,d_1$ and $i=1,\ldots,f(s,d_1)$. 
Thus,
\begin{eqnarray*}
S_{3,2} &  \le &  \sum_{d_1\mid d} \sum_{\substack{(\log x)^2< s\le x\\ \gcd(\lambda(s),d)=d_1}}f(s,d_1) \Lambda(s) \sum_{i=1}^{f(s,d_1)} \sum_{\substack{p\le x\\ p\equiv y_i(s,d_1)\pmod s}} \frac{x}{ps} \\
& \ll & x\sum_{d_1\mid d} \sum_{\substack{(\log x)^2<s\le x\\ \gcd(\lambda(s),d)=d_1}} \left(\frac{f(s,d_1)^2 \Lambda(s)\log_2 x}{s\phi(s)}+\frac{f(s,d_1)^2 \Lambda(s)\log s}{s\phi(s)}\right) \\
& \ll & \sum_{d_1\mid d} \frac{x d_1^2 (\log_2 x)^2}{(\log x)^2} =  o(d^2 x),
\end{eqnarray*}
as $x\to\infty$. So, it remains to deal with $S_{3,1}$. Using again \eqref{eq:NP}, we get
\begin{eqnarray*}
S_{4,1} & = & \sum_{d_1\mid d} \sum_{\substack{s\le (\log x)^2\\ \gcd(\lambda(s),d)=d_1}} f(s,d_1)\Lambda(s) \sum_{i=1}^{f(s,d_1)} \sum_{\substack{p\le x/s\\ p\equiv y_i(s,d_1)\pmod s}} \frac{x}{ps}\\
& = & \sum_{d_1\mid d} \sum_{\substack{s\le (\log x)^2\\ \gcd(\lambda(s),d)=d_1}} f(s,d_1)^2 \Lambda(s)\left(\frac{\log\log(x/s)}{\phi(s)}+O\left(\frac{\log s}{\phi(s)}\right)\right)\\
& = & x\left(\sum_{d_1\mid d} \sum_{\substack{s\le (\log x)^2\\ \gcd(\lambda(s),d)=d_1}} \frac{f(s,d_1)^2 \Lambda(s) \log\log (x/s)}{s\phi(s)}+O\left(\sum_{d_1\mid d} \sum_{\substack{s\le (\log x)^2\\ \gcd(\lambda(s),d)=d_1}} \frac{f(s,d_1)^2 \Lambda(s)\log s}{s\phi(s)}\right)\right).
\end{eqnarray*}
The inner sum inside the $O$ over the $s$ is convergent so that sum is $O(d^2 x)$. The remaining sum is 
\begin{eqnarray*}
x\log_2 x\sum_{d_1\mid d} \sum_{\gcd(\lambda(s),d)=d_1} \frac{f(s,d_1)^2\Lambda(s)}{s\phi(s)} & + & O\left(x\sum_{d_1\mid d} \sum_{s\le (\log x)^2} \frac{f(s,d_1)^2\Lambda(s) \log s}{s\phi(s)\log x}\right)\\
& + & 
O\left(x\sum_{d_1\mid d} \sum_{\substack{s>(\log x)^2\\ \gcd(\lambda(s),d)=d_1}} \frac{f(s,d_1)^2 \Lambda(s)}{s\phi(s)}\right).
\end{eqnarray*}
In the error terms, the first one is $O(d^2 x/\log x)=o(d^2 x)$ as $x\to\infty$ and the second one is $O(d^2 \log_2 x/(\log x)^2)=o(d^2 x)$. Finally, note that the first term above is just $c_3 x\log_2 x$.
This finishes the proof. 
\end{proof}

\section{Bad witnesses of the stronger test}\label{sect:3}

The authors of \cite{Couv.Ezo.Ler} described
the needed formalism concerning the product of pseudo-primality tests.  
In this section, we study the
average sizes of the set of bad witnesses for 
a pseudo-primality test which is the product of several
Miller-Rabin tests by the Galois test.

\subsection{Preliminaries}

Let $n>2$ be an odd composite integer such that $n-1=2^km$ with $m$ odd.
Let
$$ \mathbf{F}(n)=
\{a\pmod n : a^{n-1} -1\equiv 0 \pmod n \},$$
and
$$\mathbf{MR}(n)=\{a\pmod n : 
a^m \equiv 1 \pmod n, \ \text{ or } \
a^{2^im} \equiv -1 \pmod n \ \text{ for some } \  i<k  \}$$
be the sets of bad witnesses for the Fermat test and the Miller-Rabin test
respectively. Then $\mathbf{MR}(n)$ is a subset of $\mathbf{F}(n)$,
and the later is a subgroup of the group of units in $\ZZ/n\ZZ$. 
Note that $\#\mathbf{F}(n)=F(n)$, where $F(n)$ is given in
equation $(\ref{eq:21})$. Set
$\#\mathbf{MR}(n)=M\!R(n)$.
Let $v_p(n)$ be the exponent on $p$
in the prime factorization of $n$. For  a positive integer $k$ we let
$k':=(k-1)/2^{\nu_2(k-1)}$ denote $\text{the largest odd divisor of } k-1.$ For example, $n'=m$. We set
\begin{equation}\label{eq:10}
v(n)=\text{min}_{p|n} \{ v_2(p-1)\}, \quad
w(n)=\sum_{p|n} 1, \quad s(n)=\prod_{p|n} (n',p').
\end{equation}

Then
\begin{equation}\label{eq:3}
M\!R(n)=\left( 1+\frac{2^{v(n)w(n)}-1}{2^{w(n)}-1}\right)s(n).
\end{equation}

It is shown (\cite{Erdos-Pomerance}, proof of Theorem 5.2) that
\begin{eqnarray}
\label{eq:23}
\frac{2}{x} \sum_{\substack{1\le n\le c\\ n\equiv 1\pmod 2}} 
 \log M\!R(n) & = &\left( c_1-\frac{2\log2}{3}\right)
\log\log x+ \log c_2 +1+\log2\\ \nonumber
& & -\sum_{i\ge 1}\frac{C(2^i)}{2^i}\log2
+O\left(\frac{\log\log x}{\sqrt{\log x}} \right),
\end{eqnarray}
where
\begin{equation}\label{eq:22}
C(d)=O\left( \frac{\log d}{\varphi (d)}\right)
\text{ for any integer } d\ge 2,
\end{equation}
and $c_1, c_2$ are explicit constants defined in equation $(\ref{eq:5})$.

\subsection{Arithmetic mean of the bad witnesses of the stronger test}\label{subsect:5}

Let $x$ be a large real number. It is shown 
(\cite{Erdos-Pomerance}, proof of Theorem 5.1)
that there exist real numbers $\alpha, \beta, \varepsilon$ with
$$\beta>23/8, \ 1<\alpha<\beta, \ 0<\varepsilon<\alpha-1  \quad( \varepsilon \text{ arbitrarily small })\ \text{such that } \ M\!R(n)> x^{1-\varepsilon}/\log\log x,$$
and  a certain set $\mathbf{S}'$ of positive integers $n\le x$
such that
$$ \#\mathbf{S}'\ge x^{(\alpha-1)\alpha^{-1}+o(1)}.$$
Given a positive integer $r\ge 2$, we denote by
$F^r(n)$ and $M\!R^r(n)$ the number of bad
witnesses of the product of $r$ Fermat tests and
the product of $r$ Miller-Rabin tests
respectively. So
 $$M\!R^r(n)> \left( \frac{x^{1-\varepsilon}}{\log\log x} \right)^r,
 $$
 and
 $$\frac{1}{x} \SSIG_{n\le x}  M\!R^r(n) \ge \frac{1}{x} 
 \sum_{n\in \mathbf{S}'} x^{r-r\varepsilon+o(1)}
\ge  x^{-1+r-r\varepsilon+o(1)}
 \cdot \#\mathbf{S}' =x^{r-r\varepsilon-\alpha^{-1}+o(1)}.$$
 Since $\varepsilon>0$ is arbitrarily small and $\alpha$ is arbitrarily
 close to $\beta$, we have
$$\frac{1}{x}\SSIG_{n\le x} M\!R^r(n)\ge x^{r-\beta^{-1}+o(1)}.$$
Hence,
\begin{equation}\label{eq:2}
\frac{1}{x}\SSIG_{n\le x} M\!R^r(n)\ge x^{r-\frac{8}{23}+o(1)}.
\end{equation}

On the other hand, let $\mathbf{C}_k(x)$ denote the set of composite
number $n\le x$ such that 
$$F(n)=\varphi(n)/k \text{ for each integer } k.$$
Set $C_k(x)=\#\mathbf{C}_k(x)$. Since 
$M\!R(n)\le F(n)$, we have
$$
\SSIG_{n\le x} M\!R^r(n)  \le \SSIG_{n\le x} F^r(n) = 
\sum_k \sum_{n\in \mathbf{C}_k(x)} F^r(n)$$
$$
\le  \sum_k \sum_{n\in \mathbf{C}_k(x)}  \frac{n^r}{k^r}
$$
$$
\le \sum_{n\le x} \frac{n^r}{L^r(x)} + \sum_{k\le L(x)}  \sum_{n\in \mathbf{C}_k(x)} 
\frac{n^r}{k^r}
$$
$$
 \le  \frac{x^{r+1}}{L^r(x)} +x^r  \sum_{k\le L(x)}  \frac{C_k(x)}{k^r}.
$$

It is shown 
(\cite{Erdos-Pomerance}, proof of Theorem 2.2) that
$$
C_k(x) \le xL(x)^{-1+o(1)}\quad \text{uniformly for } k \le L(x),
$$
as $x\to\infty$. Further, 
$$\sum_{k\le x} \frac{1}{k^r}<\zeta(r)<2,\quad \text{where } \zeta(r)
 \text{ denotes the Riemann zeta function,}$$
 since $r\ge 2$. So,
$$
\frac{1}{x}\SSIG_{n\le x} M\!R^r(n)  \le x^r L(x)^{-1+o(1)},
$$
and we have proved the following theorem.

\begin{theorem}\label{thm:3}
For every integer $r \ge 2$, we have
$$
x^{r-\frac{8}{23}} <
\frac{1}{x}\SSIG_{n\le x} M\!R^r(n)  \le x^{r}L(x)^{-1+o(1)}
\ \text{ as } x \to \infty.
$$
\end{theorem}

Denoting by $\mathbf{Str}(u)$
the sets of bad witnesses for the stronger test
and $Str(n)$ its cardinality, we deduce from
Theorem \ref{thm:5}, Theorem \ref{thm:4} and
Theorem \ref{thm:3} the lower and upper bounds below.

\begin{cor}
As $x \to \infty$,
$$
x^{r+\frac{7}{23}} <
\frac{1}{x}\SSIG_{n\le x} Str(n)  \le x^{r+d}\zeta(r)L(x)^{-2+o(1)}.$$
\end{cor}

\subsection{Geometric mean of the bad witnesses of the stronger test}\label{subsect:7}

We start by studying the geometric
mean of the number of bad witnesses of a multiple-rounds
Miller-Rabin test.

\begin{theorem}\label{thm:6}
For $c_1$ given by $(\ref{eq:5})$ and $c_4$ defined in $(\ref{eq:7})$ below,
we have
$$\left( \prod_{1< n\le x} M\!R^r(n) \right)^{2/x}
=c_4(\log x)^{r(c_1-\frac{2\log2)}{3})}+ O\left(\frac{\log\log x}{\sqrt{\log x}} \right)$$
as $x\to\infty$.
\end{theorem}

\begin{proof}
From Equation $(\ref{eq:23})$,
we have
$$
\frac{2}{x} \sum_{\substack{1\le n\le x\\ n\equiv 1\pmod 2}} \log (M\!R^r(n))=r\left( c_1-\frac{2\log2}{3}\right)
\log\log x+ \log (c_2^r) +r+\log2^r $$
$$ -r\sum_{i\ge 1}\frac{C(2^i)}{2^i}\log2
+O\left(\frac{\log\log x}{\sqrt{\log x}} \right).
$$
The desired asymptotic formula follows by setting
\begin{equation}\label{eq:7}
c_4= \left(\frac{2ec_2}{2^{\sum_{i\ge 1}{C(2^i)}{2^i}}}\right)^r.
\end{equation}
\end{proof}

Since $Gal(n)\ge F(n)$, it follows that 
the geometric mean value of $Str(n)$ is at least as large as
the geometric mean value of $M\!R^r(n)\times F(n)$.
%From the study made in Subsection \ref{subsect:3},
%We have the following result.
By \cite[Theorem 3.1]{Erdos-Pomerance} and Theorem \ref{thm:6},
we know that $
\Big(\prod_{1<n\le x}M\!R^r(n)\times F(n)\Big)^{2/x}$ equals
$$
\Bigg(c_4 (\log x)^{r(c_1-\frac{2\log2}{3})} + 
O\left(\frac{\log\log x}{\sqrt{\log x}} \right)\Bigg)\times
\Bigg(c_2^2(\log x)^{2c_1}+O\Big((\log x)^{2c_1-1}\Big) \Bigg).
$$
This means that 
$$
\Big(\prod_{1<n\le x}M\!R^r(n)\times F(n)\Big)^{2/x}
=c_4c_2^2 (\log x)^{(r+2)c_1-\frac{2r\log2}{3}} + 
O\Big((\log x)^{(r+2)c_1-(\frac{2r\log2}{3}+1)}\Big).$$

On the other hand, the geometric mean of $Str(n)$
is at most as large as the geometric mean of
$\Big(M\!R^r(n)\times H(n)\Big)^{2/x}$ because $Gal(n)$
divides $H(n)$. By Theorem \ref{thm:1} and Theorem \ref{thm:6},
we have that
$\Big(\prod_{1<n\le x}M\!R^r(n)\times H(n)\Big)^{2/x}$ 
is of order 
$$
\asymp_d (\log x)^{r\left(c_1-\frac{2\log 2}{3}\right)+2c_3},
$$
where the constant implied by the above sign are of size $e^{O(d^4)}$.

\section{Conclusion and Perspectives}

This work describes the arithmetic and geometric mean values of 
the number of bad witnesses $Str(n)$ when testing the primality of
$n$ by using a pseudo-primality test
which is the product of a multiple-rounds Miller-Rabin
test by the Galois test. We actually compute functions which
squeeze these mean values.
These bounds can still be improved. On the other hand,
one could also study the normal order of $Str(n)$,
that is its behaviour on
a set of natural numbers of asymptotic density $1$.
We thus have at least two alternatives for future work.\\

\bibliographystyle{plain}

\bibliography{Bad_Witnesses_composite_Number.bib}

\end{document}